\documentclass{amsart}
\usepackage{amsmath,amssymb, amscd}

\newtheorem{prop}{{\bf Proposition}}[section]
\newtheorem{coro}[prop]{{\bf Corollary}}
\newtheorem{lemma}[prop]{{\bf Lemma}} 
\newtheorem{theor}[prop]{{\bf Theorem}} 
\newtheorem{ex}[prop]{{\bf Example}} 
\newtheorem{remark}[prop]{{\bf Remark}}

\def\C{{\mathbb C}}
\def\N{{\mathbb N}}
\def\Leib{{\rm Leib}}
\def\sp{{\rm span}}
\DeclareUnicodeCharacter{2212}{-}
\begin{document}
\title[On the subalgebra lattice of a Leibniz algebra]{On the subalgebra lattice of a Leibniz algebra}
\author{Salvatore Siciliano}
\address{Dipartimento di Matematica e Fisica ``Ennio De Giorgi", Universit\`{a} del Salento,
Via Provinciale Lecce--Arnesano, 73100--Lecce, Italy}
\email{salvatore.siciliano@unisalento.it}
\author{David A. TOWERS}
\address{Lancaster University\\
Department of Mathematics and Statistics \\
LA$1$ $4$YF Lancaster\\
ENGLAND}
\email{d.towers@lancaster.ac.uk}

\subjclass[2010]{17A32, 17B05, 17B30,   17B60.}
\keywords{Lie Algebras; Frattini Ideal;  Nilpotent; Solvable; cyclic, upper semi-modular, lower semi-modular, extraspecial Leibniz algebra. 
Supersolvable.}

\begin{abstract} In this paper we begin to study the subalgebra lattice of a Leibniz algebra. In particular, we deal with Leibniz algebras whose subalgebra lattice is modular, upper semi-modular, lower semi-modular, distributive, or dually atomistic. The fact that a non-Lie Leibniz algebra has fewer one-dimensional subalgebras in general results in a number of lattice conditions being weaker than in the Lie case.
\end{abstract}

\maketitle

\section{Introduction}
An algebra $(L,[\cdot,\cdot])$ over a field $F$ is called a {\em Leibniz algebra} if, for every $x,y,z \in L$, we have
\[  [x,[y,z]]=[[x,y],z]-[[x,z],y].
\]
In other words, the right multiplication operator $R_x : L \rightarrow L$, $y\mapsto [y,x]$, is a derivation of $L$. As a result, such algebras are sometimes called {\it right} Leibniz algebras, and there is a corresponding notion of {\it left} Leibniz algebra. Every Lie algebra is a Leibniz algebra and every Leibniz algebra satisfying $[x,x]=0$ for every element is a Lie algebra. 
\par

Leibniz algebras were first considered by Bloh in \cite{bloh} and Loday in \cite{lod}, and nowadays they play an important role in several areas of mathematics such as homological algebra, algebraic $K$-theory, differential geometry,  algebraic topology, noncommutative geometry, etc. As a result, the theory of these algebraic structures has been developing intensively in the last three decades and many important theorems for Lie algebras have been considered in the more general context of Leibniz algebras.

Now, the set of subalgebras of a nonassociative algebra forms a lattice under the operations $\vee$ and $\wedge$,  where the join $\vee$ of two subalgebras is the subalgebra generated by their set-theoretic union, and the meet $\wedge$ is the usual intersection. The relationship between the structure of a Lie algebra $L$ and that
of the lattice ${\mathcal L}(L)$ of all subalgebras of $L$ has been studied by many authors. Much is known about modular subalgebras
(modular elements in ${\mathcal L}(L)$) through a number of investigations including \cite{as1,g2,g3,v6,v4,v5}. Other lattice conditions, together with
their duals, have also been studied. These include semimodular, upper semimodular, lower semimodular, upper modular, lower modular
and their respective duals. For a selection of results on these conditions see \cite{bv14, gein, gv10, kol, l11,  t13, t15, sch, v12, v9}.
\par

The subalgebra lattice of a Leibniz algebra, however, is rather different; in a Lie algebra every element generates a one-dimensional subalgebra, whereas in a Leibniz algebra elements can generate subalgebras of any dimension. Thus, one could expect different results to hold for Leibniz algebras and, as we shall see, this is indeed the case. In the second section we show that cyclic Leibniz algebras are determined by their subalgebra lattice. In Section 3 we classify Leibniz algebras over a field of characteristic zero in which every subalgebra is an intersection of maximal subalgebras; an extra family arises in the non-Lie case.
\par

In Section 4 we study upper semi-modular Leibniz algebras. There turn out to be many more of these than in the Lie algebra case; in particular, all nilpotent cyclic Leibniz algebras and a subset of the extraspecial Leibniz algebras introduced in \cite{kss} belong to this class. Section 5 is devoted to lower semi-modular Leibniz algebras. The situation here turns out to be very similar to the Lie algebra case. The final section deals with modular Leibniz algebras. We see that this is a much larger class than in the Lie algebra case; in particular, any Leibniz which is of the form $E\oplus C$, where $E$ is an upper semi-modular extraspecial or a nilpotent cyclic Leibniz algebra and $C$ is a central ideal is modular. However, these do not exhaust all of the modular Leibniz algebras even of dimension four.
\par

We fix some notation and terminology. Unless otherwise stated, throughout the paper all algebras are assumed to be finite-dimensional. Algebra direct sums will be denoted by $\oplus$, whereas vector space direct sums will be denoted by $\dot{+}$.
\par

Let $L$ be a Leibniz algebra over a field $F$. For a subset $S$ of $L$, we denote by $\langle S \rangle$ the subalgebra generated by $S$.   
The {\it Leibniz kernel} is defined as $\Leib (L)= \sp\{x^2 \vert \, x\in L\}$. Note that $\Leib (L)$ is the smallest ideal of $L$ such that $L/\Leib(L)$ is a Lie algebra.
Also, $[L,\Leib(L)]=0$.
\par

We define the following series:
\[ L^1=L,\, L^{k+1}=[L^k,L] \quad (k\geq 1) 
\]
and
\[ L^{(0)}=L, \, L^{(k+1)}=[L^{(k)},L^{(k)}] \quad (k\geq 0).
\]
Then $L$ is {\em nilpotent of class n} (respectively {\em solvable of derived length n}) if $L^{n+1}=0$ but $L^n\neq 0$ (respectively $ L^{(n)}=0$ but $L^{(n-1)}\neq 0$) for some $n \in \N$. It is straightforward to check that $L$ is nilpotent of class $n$ precisely when every product of $n+1$ elements of $L$, no matter how associated, is zero, but some product of $n$ elements is non-zero (see e.g. Proposition 5.4 of \cite{feld}). The {\em nilradical}, $N(L)$, (respectively {\em radical}, $R(L)$) is the largest nilpotent (respectively solvable) ideal of $L$. 
\par

The {\em Frattini ideal} of $L$, $\phi(L)$, is the largest ideal contained in all maximal subalgebras of $L$; if $\phi(L)=0$ we say that $L$ is {\em $\phi$-free}. The {\em right centraliser} of a subalgebra $U$ of $L$ is the set $C_L^r(U)=\{ x\in L \mid [U,x]=0\}$. It is easy to check that if $U$ is an ideal of $L$ then so is $C_L^r(U)$. We denote by $\rm{soc}(L)$ the \emph{socle} of $L$, that is, the sum of all minimal ideal of $L$. 
\par

\section{Cyclic Leibniz algebras}
 A Leibniz algebra $L$ is said to be {\it cyclic} if it is generated by a single element. In this case $L$ has a basis $a,a^2, \ldots, a^n$ $(n > 1)$ and product $[a^n,a]=\alpha_2a^2+ \ldots + \alpha_na^n$.  Let $T$ be the matrix for $R_a$ with respect to the above basis. Then $T$ is the
companion matrix for $p(x) =  x^n - \alpha_n x^{n-1} - \cdots - \alpha_2 x= p_1(x)^{n_1} \cdots p_s(x)^{n_s}$, where the
$p_j$'s are the distinct irreducible factors of $p(x)$. We shall speak of these factors as being the distinct irreducible factors {\it associated} with $L$. The purpose of this short section is to characterise cyclic Leibniz algebras in terms of their subalgebra lattice. We shall need the following result, which is proved in \cite[Corollary 4.3]{batten}.

\begin{theor}\label{t:cyclic} The maximal subalgebras of the cyclic Leibniz algebra $L$ are precisely the null spaces of $r_j(R_a)$, where $r_j(x)=p(x)/p_j(x)$ and the $p_j(x)$ are the distinct irreducible factors associated with $L$ for $j=1,\ldots, s$.
\end{theor}

\begin{theor} Let $L$ be a Leibniz algebra over a field $F$ with $\vert F \vert >s$. Then $L$ has precisely $s$ maximal subalgebras if and only if it is cyclic with $s$ distinct associated irreducible factors.
\end{theor}
\begin{proof} Necessity follows from Theorem \ref{t:cyclic}.
 To prove sufficiency, let $M_1,\ldots, M_s$ denote the maximal subalgebras of $L$. Then $\cup_{i=1}^s M_i$ cannot be the whole $L$, since $\vert F \vert >s$. Choose $x\in L\setminus \cup_{i=1}^s M_i$. Then $x$ generates $L$.
\end{proof}

\begin{coro} Let $L$ be a Leibniz algebra over an infinite field. Then $L$ is cyclic if and only if it has only finitely many maximal subalgebras.
\end{coro}

\section{Dually atomistic Leibniz algebras}
We say that a Leibniz algebra $L$ is {\em dually atomistic} if every subalgebra of $L$ is an intersection of maximal subalgebras of $L$. It is easy to see that if $L$ is dually atomistic then so is every factor algebra of $L$, and if $L$ is dually atomistic then it is $\phi$-free. Dually atomistic Lie algebras over a field of characteristic zero were classified in \cite{sch}. An extra family arises for Leibniz algebras.

\begin{lemma}\label{l:nilrad} Let $L$ be dually atomistic and let $N$ be the nilradical
  of $L$.  Then
\begin{enumerate}
\item[(i)] $M \cap N$ is an ideal of $L$ for every maximal subalgebra
  $M$ of $L$; and
\item[(ii)] every subspace of $N$ is an ideal of $L$, and $N^{2} = 0$.
\end{enumerate}
\end{lemma}

\begin{proof} (i) The result is clear if $N \subseteq M$, so suppose that
  $N \not \subseteq M$.  Then $L = N + M$ and 
\begin{align*}
[L,N \cap M]  &= [N + M,N \cap M]
 \subseteq  N^{2} + M^{2} \cap N\\ 
&\subseteq  N \cap \phi (L) + M \cap N \subseteq  N \cap M,
\end{align*}
using Theorem 6.5 of \cite{frat} (or Theorem 2.4 of \cite{BBHISS}).
\par
(ii) We have $N^{2} \subseteq \phi (L) = 0$, so every subspace of $N$
is a subalgebra of $L$.  Let $S$ be any subspace of $N$.  Then
$$
S  =  S \cap N   = \big(\bigcap_{M \in \mathcal{M}} M\big) \cap N  = \bigcap_{M \in \mathcal{M}} (M \cap N),
$$
where $\mathcal{M}$ is the set consisting of all maximal subalgebras of $L$
containing $S$.  Therefore, $S$ is an intersection of ideals of $L$,
by (i), and so is itself an ideal of $L$.
\end{proof}

\begin{prop}\label{p:solvable}  Let $L$ be a solvable Leibniz algebra over
  any field $F$.  Then $L$ is dually atomistic if and only if one of the following conditions holds:
\begin{enumerate}
\item[(i)] $L$ is an abelian or almost abelian Lie algebra;
\item[(ii)] $L=\Leib(L)\dot{+} Fv$, where $v^2=0$ and $[e,v]=e$ for every $e\in \Leib(L)$. 
\end{enumerate}
\end{prop}

\begin{proof} Obviously, we can assume that $L$ is not abelian.
We first prove necessity. 
Let $N$ be the nilradical of $L$. Then $N$ is abelian by Lemma \ref{l:nilrad}. As $L$ is $\phi$-free, from \cite[Theorems 2.4 and 2.6]{BBHISS} it follows that $L=N\dot{+} V$ for some subalgebra $V$ of $L$. As $L$ is solvable, all minimal ideals of $L$ are abelian.  Moreover, by Lemma \ref{l:nilrad}, every subspace of $N$ is an ideal of $L$. Therefore, by \cite[Theorem 2.4]{BBHISS} and \cite[Lemma 1.9]{barnes_th}, we deduce that ${\rm soc}(L)=N=C_L^r(N)$. Note that the map $$\theta : L \rightarrow \mathrm{Der} (Fn), \; x \mapsto R_x\mid_{Fn}$$ is an endomorphism with kernel $C_L^r(Fn)$ and so $C_L^r(Fn)$ has codimension at most one in $L$. Let $n_1,n_2\in N$ such that $[n_i,L]\neq 0$, $i=1,2$. Since every subspace of $N$ is an ideal of $L$, all elements of $L$ act by scalar multiplication on $N$. It follows that $C_L^r(Fn_1)=C_L^r(Fn_2)=C_L^r(N)=N$, which allows us to conclude  that $V$ is one-dimensional.

 Let $ v\in V$, $v\neq 0$.
As $[N,v]\neq 0$, there exists $n\in N$ such that $[n,v]=\lambda n$ for some $\lambda \neq 0$. Replacing $v$ by $\lambda^{-1} v$ we can assume that $[n,v]=n$, so that $R_v$ acts as the identity map on $N$.   Now, if 
$[v,n]=\alpha n$, then we have
\[ \alpha^2n=[v,[v,n]]=[v^2,n]-[[v,n],v]=-\alpha n,
\] so $\alpha=0,-1$.

Put $I=\{ n\in N \mid [v,n]=0\}$, $U=\{ n\in N \mid [n,v]=-[v,n]=n\}$. Then $I$ and $U$ are ideals of $L$  and $N=I\oplus U$. Clearly, $\Leib(L)\subseteq I$. But, also, if $n\in I$ then we have $n=(v+n)^2\in \Leib(L)$. Finally, as every subspace of $N$ is an ideal of $L$,  we must have  either $I=0$ or $U=0$,  and the assertion follows. 
\par

We now prove the converse. If $L$ is an almost abelian Lie algebra, then, by \cite[Proposition 1.1(3)]{kol}, every subspace of $L$ is a subalgebra, so $L$ is obviously dually atomistic. Suppose next that $L$ satisfies condition (ii) of the statement. Note that $B+Fv$ is a maximal subalgebra of $L$ for any maximal subspace $B$ of $\Leib(L)$. Let $H$ be a subalgebra of $L$. Let $x\in H$  and write $x=a+\lambda v$, for some $a\in \Leib(L)$ and $\lambda \in F$. One has $(a+\lambda v)^2= \lambda a$, hence $a\in H$. This proves that either $H$ is contained in $\Leib(L)$ or is of the form $H=A+ Fv$ for some subalgebra $A$ of $\Leib(L)$. In the latter case, we clearly have 

$$
H= \Big( \bigcap_{B\in {\mathcal M_2}} (B+Fv) \Big),
$$
where ${\mathcal M_1}$ denote the set of all maximal subspaces  of $\Leib(L)$ containing $A$.
On the other hand, if $H\subseteq \Leib(L)$, then  we have
$$
H=\Leib(L) \cap \Big( \bigcap_{B\in {\mathcal M_1}} (B+Fv) \Big),
$$
where ${\mathcal M_1}$ denotes the set of all maximal subspaces  of $\Leib(L)$ containing $H$. This completes the proof.   
\end{proof}
\medskip


The combination of Theorem \ref{p:solvable} with the next result yields a full characterization of dually atomistic Lie algebras over fields of characteristic zero.

\begin{prop} Let $L$ be a dually atomistic Leibniz algebra over a field of characteristic zero. Then $L$ is a three-dimensional non-split simple Lie algebra or is solvable.
\end{prop}
\begin{proof} We have that $L/\Leib(L)$ is a solvable or three-dimensional non-split simple Lie algebra by \cite[Lemmas 1 and 2]{sch}. In the former case, $L$ is solvable, so suppose that the latter case holds. Then, by \cite[Theorem 1]{barnes_BAustr}, we have $L=S\dot{+}I$, where $S$ is a three-dimensional non-split simple Lie algebra and $I=$Leib$(L)$. Moreover, every subspace of $I$ is an ideal of $L$, by Lemma \ref{l:nilrad} and the fact that $L$ is $\phi$-free. Let $x\in I$. Then $[S,x]=0$ and, for all $s_1, s_2\in S$, $[x,s_1]=\lambda x$, $[x,s_2]=\mu x$ for some $\lambda, \mu \in F$, whence
\[ [x,[s_1,s_2]]= [[x,s_1],s_2]-[[x,s_2],s_1]=\lambda \mu x - \mu \lambda x = 0.
\] It follows that $[x,S]=[x,[S,S]]=0$ and $L=S\oplus I$.
Therefore $L$ is a Lie algebra, and it follows from \cite[Lemma 1]{sch} that $I=0$, as desired.
\end{proof}

\section{Upper semi-modular Leibniz algebras}
Let $L$ be a Leibniz algebra. A subalgebra $U$ of $L$ is called {\em upper semi-modular} if $U$ is a maximal subalgebra of $\langle U,B\rangle$ for every subalgebra $B$ of $L$ such that $U\cap B$ is maximal in $B$. We say that $L$ is {\it upper semi-modular} if every subalgebra of $L$ is upper semi-modular in $L$. The Lie algebras in this class were classified in \cite{kol}. There are many more Leibniz algebras in this class.

We first establish the following result, which characterize Leibniz algebras having a distributive lattice of subalgebras.  
\begin{prop} The only non-Lie distributive Leibniz algebras are the two-dimensional cyclic Leibniz algebras.
\end{prop}
\begin{proof} Let $L$ be a distributive Leibniz algebra with $\Leib(L)\neq 0$. Then $L/\Leib(L)$ must be a distributive Lie algebra and so is one-dimensional.
Now, if $L$ has dimension greater than two, then $\Leib(L)$ is an abelian Lie subalgebra of dimension greater than one, and so $L$ is not distributive. 
 Hence,  $L$ is  two-dimensional cyclic, and it has a basis $x, x^2$ with multiplication $[x^2,x]=0$ or $[x^2,x]=x^2$. In the former case, the only one-dimensional subalgebra is $Fx^2$, and in the latter case the only one-dimensional subalgebras are $Fx^2$ and $F(x-x^2)$. In either case, $L$ is distributive.
\end{proof}
\medskip

The following is a straightforward exercise.

\begin{lemma}\label{l:factor} Let $U$ be a subalgebra of the Leibniz algebra $L$, and let $I$ be an ideal of $L$ contained in $U$. Then $U$ is upper semi-modular in $L$ if and only if $U/I$ is upper semi-modular in $L/I$.
\end{lemma}

\begin{lemma}\label{l:nilp} Let $L$ be a nilpotent cyclic Leibniz algebra. Then $L$ is upper semi-modular.
\end{lemma}
\begin{proof} Let $L$ be generated by $a$ where $[a^n,a]=0$. Suppose that $U$ is a subalgebra of $L$ which is not in $L^2$. Let $0\neq u \in U$. Then $u=a+\sum_{i=2}^n\lambda_ia^i\in U$ for some $\lambda_i\in F$. But then it is easy to see that $u,u^2,\ldots ,u^n$ are linearly independent, so $U=L$. Hence, all proper subalgebras of $L$ are inside $L^2$, which is abelian. The result follows.
\end{proof}

\begin{prop}\label{upper} Let $L$ be a non-Lie $\phi$-free Leibniz algebra over any field $F$. Then $L$ is upper semi-modular if and only $L=\Leib(L)\dot{+} Fv$, where $\Leib(L)=Fe_1+\cdots + Fe_r$  and $[e_i,v]=e_i$ for $1\leq i\leq r$, all other products being zero.  
\end{prop}
\begin{proof} Suppose first that $L$ is upper semi-modular. Then $L=N\dot{+}V$ where $N$ is the nilradical, is abelian, and $V$ is a subalgebra of $L$, by Theorems 2.4 and 2.6 of \cite{BBHISS}.  Let $v\in V$, $n\in N$. Then $[v,v]\in V\cap$ Leib$(L)\subseteq V\cap N=0$ and $Fv$ covers $Fv \cap Fn=0$, so $Fn$ is covered by $\langle n,v\rangle  $. It follows that $[n,v], [v,n]\in Fn$. As this is true for all $v\in V$, every subspace of $N$ is an ideal of $L$. Then, as in Proposition \ref{p:solvable}, $L$ is abelian or an almost abelian Leibniz algebra. Now suppose that $V=U\dot{+}Fv$. Then $Fv$ covers $F(a+u)\cap Fv$, where $a\in \Leib(L)$, $u\in U$. Hence $Fa$ is covered by $\langle a+u,v\rangle  =Fa+Fu+Fv$. It follows that $u=0$ and we have the multiplication claimed. 

Conversely, suppose that $L$ has the form given. Then the subalgebras of $L$ are of the form $B$ or $Fv+B$ where $B$ is a subalgebra of $\Leib(L)$. If $B,C\subseteq \Leib(L)$, they clearly satisfy the upper semi-modular condition. Consider $B$ and $Fv+C$. Then $B\cap (Fv+C)=B\cap C$ and $\langle B,Fv+C\rangle  =Fv+B+C$. If $B$ covers $B\cap C$, we have $B=B\cap C+Fb$ for some $b\in B$ and $\langle B,Fv+C\rangle  =Fv+C+Fb$ which covers $Fv+C$. If $Fv+C$ covers $B\cap C$ then $B\cap C=C$, whence $C\subseteq B$ and $\langle B,Fv+C\rangle  =Fv+B$ which covers $B$. 
\medskip

Finally consider $Fv+B$ and $Fv+C$. Then $(Fv+B)\cap (Fv+C)=Fv+B\cap C$ and $\langle Fv+B,Fv+C\rangle  =Fv+B+C$. If $Fv+B$ covers $Fv+B\cap C$, then we have that $B=B\cap C+ Fb$ and $Fv+B+C=Fa+C+Fb$ which covers $Fv+C$. By symmetry this is sufficient in this case. Hence $L$ is upper semi-modular.
\end{proof}

\begin{coro}\label{usmac} Let $L$ be an upper semi-modular Leibniz algebra over an algebraically closed field. Then $L$ is supersolvable.
\end{coro}
\begin{proof} If $L$ is a Lie algebra, then the assertion follows from \cite[Corollary 2.2]{kol}. On the other hand, if $L$ is non-Lie, then by Proposition \ref{upper} we have that $L/\phi(L)$ is supersolvable, and hence so is $L$, by \cite[Theorems 3.9 and 5.2]{barnes}.
\end{proof}

The consideration of the three-dimensional non-split simple Lie algebras shows the hypothesis that the ground field is algebraically closed cannot be dropped in Corollary \ref{usmac}.

\begin{lemma}\label{nilp} Let $L$ be an upper semi-modular nilpotent Leibniz algebra. Then 
\begin{itemize}
\item[(i)] $[$Leib$(L),x]\subseteq \langle x\rangle  \cap \Leib(L)$ for all $x\in L$;
\item[(ii)]  $[x^i,y]\in Fx^{i+1}+\cdots +Fx^{n+1}$, for all $x\not \in  \Leib(L)$, $y\in L$, $2\leq i\leq n$, where $x^{n+1}=0$, $x^n\neq 0$. In fact, $R_y  \mid_{\langle x\rangle  \cap \Leib(L)}$ has matrix
$$
\begin{pmatrix}
0&0&\ldots&0&0\\
\alpha_3&0&\dots&0&0\\
\vdots&\vdots&\vdots&\vdots&\vdots\\
\alpha_{n-1}&\alpha_{n-2}&\ldots&0&0\\
\alpha_n&\alpha_{n-1}&\dots&\alpha_3&0
\end{pmatrix};
$$

\item[(iii)] $\langle x\rangle  \cap \Leib(L)$ is an ideal of $L$ for all $x\not \in \Leib(L)$;
\item[(iv)] Let $J=\{x\in L \mid x^2=0\}$. Then $J$ is an abelian ideal of $L$;
\item[(v)] $[J,x]\subseteq \langle x\rangle  \cap \Leib(L)$ and $[x,J]\subseteq Fx^n$ for all $x\in L$,  where $x^{n+1}=0$, $x^n\neq 0$.
\end{itemize}
\end{lemma}
\begin{proof} (i) Let $y\in L$. If $y^2\in \langle  x\rangle  $ then $[y^2,x]\in  \langle x\rangle  \cap \Leib(L)$.  Hence suppose that $y^2 \not \in \langle x\rangle  $. Then $Fy^2\cap \langle x\rangle  =0$ is covered by $Fy^2$, so $\langle x\rangle  $ is covered by $\langle y^2,x\rangle  =\langle x\rangle  +Fy^2$. Since $L$ is nilpotent, this implies that $[y^2,x]\in \langle x\rangle  \cap \Leib(L)$ and the result follows.
\par

\noindent (ii) As $L/\Leib(L)$ is a nilpotent upper semi-modular Lie algebras, it follows from Theorem 2.2 of \cite{kol} that $L^2\subseteq \Leib(L)$.  Consequently, we have $$[x^2,y]=[x,[x,y]]+[[x,y],x]=[[x,y],x]\in Fx^2+\cdots +Fx^{n+1}$$ by (i). Suppose the result holds for $i=k$ where $k\geq 1$. Then $$[x^{k+1},y]=[[x^k,x],y]=[x^k,[x,y]]+[[x^k,y],x]\in Fx^{k+1}+\cdots +Fx^{n+1},$$ which gives the result. 
\par

\noindent (iii) This is apparent from (ii).
\par

\noindent (iv) Let $x,y\in J$. Then $Fx\cap Fy=0$ is covered by $Fx$, so $\langle x,y\rangle  =Fx+Fy$ and $[x,y]=[y,x]=0$, whence $J$ is an abelian subalgebra of $L$. Moreover, by Theorem 2.2 of \cite{kol} we have $L^2\subseteq \Leib(L)\subseteq J$, so $J$ is an ideal of $L$.
\par

\noindent[(v)] Let $y\in J$. If $y\in \langle x\rangle  $, then $[y,x],[x,y]\in \langle x\rangle  \cap \Leib(L)$, so suppose that $y\not \in \langle x\rangle  $. Then $\langle x\rangle  \cap Fy=0$ is covered by $Fy$, so $\langle x,y\rangle  =\langle x\rangle  +Fy$. It follows that  $[y,x],[x,y]\in \langle x\rangle  \cap  \Leib(L)$ again, which gives the first inclusion. Now
$$ 0=[x^2,y]=[x,[x,y]]+[[x,y],x]=[[x,y],x],$$ so $[x,y]\in Fx^n$, yielding the second inclusion.
\end{proof}

Following \cite{kss}, we say that a Leibniz algebra $L$ is {\em extraspecial} if $\dim Z(L)=1$ and $L/Z(L)$ is abelian. Then we have the following result.

\begin{prop}\label{p:esp} An extraspecial Leibniz algebra $L$ is upper semi-modular if and only if $J=\{x\in L \mid x^2=0\}$ is an abelian ideal of $L$. 
\end{prop}
\begin{proof} The necessity follows from Lemma \ref{nilp}(iv). Conversely, suppose that $J$ is an abelian ideal of $L$ and let $U,B$ be subalgebras of $L$ for which $U\cap B$ is a maximal subalgebra of $B$. We need to show that $U$ is a maximal subalgebra of $\langle U,B\rangle  $. We consider several cases. 
\par

Suppose first that $U\subseteq J$. If $B\subseteq J$ then the result is clear. Therefore we can suppose that $B \not \subseteq J$.  Then there exists $b\in B$ such that $[b,b]\neq 0$, which implies that $Z(L)\subseteq B$. It follows that $B=U\cap B+Z(L)$ and so $\langle U,B\rangle  =\langle U,Z(L)\rangle  =U+Z(L)$ which covers $U$.
\par

Suppose $U\not \subseteq J$, so that $Z(L) \subseteq U$. Then $U\cap (B+Z(L))=U\cap B+Z(L)$ so
\[ \frac{U}{Z(L)}\cap \frac{B+Z(L)}{Z(L)} \hbox{ is covered by } \frac{B+Z(L)}{Z(L)}.
\] It follows that
\[ \left \langle \frac{U}{Z(L)},\frac{B+Z(L)}{Z(L)} \right \rangle = \frac{\langle U,B \rangle}{Z(L)} \hbox{ covers } \frac{U}{Z(L)}.
\] Hence $\langle U,B \rangle$ covers $U$, completing the proof.
\end{proof}

\section{Lower semi-modular Leibniz algebras}
A subalgebra $U$ of $L$ is called {\em lower semi-modular} in $L$ if $U\cap B$ is maximal in $B$ for every subalgebra $B$ of $L$ such that $U$ is maximal in $\langle U,B\rangle$. We say that $L$ is {\it lower semi-modular} if every subalgebra of $L$ is lower semi-modular in $L$. 
\par

If $U$, $V$ are subalgebras of $L$ with $U\subseteq V$, a {\em J-series} (or {\em Jordan-Dedekind series}) for $(U,V)$ is a series
\[ U=U_0\subset U_1\subset \ldots \subset U_r=V
\] of subalgebras such that $U_i$ is a maximal subalgebra of $U_{i+1}$ for $0\leq i \leq r-1$. This series has {\em length} equal to $r$. We shall call $L$ a {\em J-algebra} if, whenever $U$ and $V$ are subalgebras of $L$ with $U\subseteq V$, all $J$-series for $(U,V)$ have the same finite length, $d(U,V)$. Put $d(L)=d(0,L)$. The following is proved in similar manner to \cite[Lemma 5]{gein}.

\begin{prop}\label{p:equiv} For a solvable Leibniz algebra the following are equivalent:
\begin{itemize}
\item[(i)] $L$ is lower semi-modular;
\item[(ii)] $L$ is a $J$-algebra; and
\item[(iii)] $L$ is supersolvable. 
\end{itemize}
\end{prop}
\begin{proof} (i)$\Rightarrow$(ii): This is a purely lattice theoretic result (see \cite[Theorem V3]{birkhoff}).
\par

\noindent (ii)$\Rightarrow$(iii): Since $L$ is solvable, all chains from $0$ to $L$ will have length $\dim L$, so all maximal subalgebras have codimension one in $L$. But then $L$ is supersolvable, by \cite[Corollary 3.10]{barnes}.

\par

\noindent (iii)$\Rightarrow$(i) Let $L$ be a supersolvable Leibniz algebra and let $U,B$ be subalgebras of $L$ such that $U$ is maximal in $\langle U,B\rangle$. Then $U$ has codimension 1 in $\langle U,B\rangle$, so $\langle U,B\rangle =U+B$. But now $\dim(B/(U\cap B)=\dim((U+B)/U)=1$, whence $U\cap B$ is maximal in $B$.
\end{proof}

The following results follow from the corresponding results in \cite{gein}.

\begin{theor}\label{t:J}(cf. \cite[Theorem 2]{gein}) Let $L$ be a Leibniz algebra over a field $F$ of characteristic $0$. Then $L$ is a $J$-algebra if and only if $L=R\oplus S$ where the radical $R$ is supersolvable and $S$ is a semisimple Lie algebra which is a $J$-algebra. Moreover, if $S$ is lower semi-modular then so is $L$.
\end{theor}
\begin{proof} Let $L$ be a $J$-algebra. By Levi's Theorem (\cite{barnes_BAustr}), $L=R\dot{+} S$, where $R$ is the radical and $S$ is a semisimple Lie subalgebra of $L$. Moreover, $R$ is supersolvable and $[S,R]+[R,S]\subseteq \Leib(L)$, by Proposition \ref{p:equiv}.  We claim that $[R,S]=[S,R]=0$.
\par

Let $L$ be a minimal counter-example. Suppose that $R$ is not a minimal ideal of $L$.  Let $R/K$ be a chief factor of $L$. Then $R=K+Fr$ for some $r\in R$, and the minimality implies that $[K,S]=[S,K]=0$ and $[r,S]+[S,r]\subseteq K$. But then $$[S,r]=[S^2,r]\subseteq [S,[S,r]]+[[S,r],S]\subseteq [S,K]+[K,S]=0.$$
Also, $$[r,S]=[r,S^2]\subseteq [[r,S],S]\subseteq [K,S]=0.$$ Hence $R= \Leib(L)$ is a minimal ideal of $L$ and $S$ is a maximal subalgebra of $L$. Let $M$ be a maximal subalgebra of $S$. Then $d(L)=d(M)+2$. But $R+M$ is a maximal subalgebra of $L$, so $d(L)=d(R+M)+1=d(R)+d(M)+1$ since $R+M$ is a $J$-algebra. It follows that $d(R)=1$ so $R=Fr$ is one-dimensional. As in the first paragraph of Proposition \ref{p:solvable} we have that $C_L^r(R)$ has codimension at most 1 in $L$. Suppose by contradiction that $C_L^r(R)\neq L$. As $R\subseteq C_L^r(R)$, we have that $C_L^r(R)\cap S$ is an ideal of codimension 1 in $S$. Then, by the Weyl's Theorem, we have $L=C_L^r(R)\oplus I$, where $I$ is a one-dimensional ideal of $S$, which is not possible as $S$ is semisimple.  Thus $C_L^r(R)= L$, in particular $[R,S]=0$.  But $[S,R]=[S,\Leib(L)]=0$, a contradiction. This establishes the claim.
\par

The converse is as in \cite[Theorem 2]{gein}.
\end{proof}

\begin{coro}(cf. \cite[Corollary 1]{gein}) Let $L$ be a Leibniz algebra over an algebraically closed field $F$ of characteristic zero. Then the following are equivalent:
\begin{itemize}
\item[(i)] $L$ is lower semi-modular;
\item[(ii)] $L$ is a $J$-algebra; and
\item[(iii)] $L=R\oplus S$ where the radical $R$ is supersolvable and $S=\mathrm{sl}_2(F)$ or $S=0$. 
\end{itemize}
\end{coro}

\begin{theor}\label{t:lm}(cf. \cite[Theorem 3]{gein}) Let $L$ be a Leibniz algebra over a field $F$ of characteristic zero. Then $L$ is lower semi-modular if and only if $L=R\oplus S_1\oplus \cdots \oplus S_n$ where the radical $R$ is supersolvable and the $S_i$ are mutually non-isomorphic three-dimensional simple algebras for $1\leq i\leq n$, and also $n\leq 1$ when $\sqrt(F)\subseteq F$ and the $S_i$ are indecomposable if $\sqrt(F)\not \subseteq F$.
\end{theor}
\begin{proof} The necessity follows easily from Proposition \ref{p:equiv}, Theorem \ref{t:J} and \cite[Theorem 3]{gein}.
\par

The converse follows from Theorem \ref{t:J} and the fact that $S=S_1\oplus \cdots \oplus S_n$ is lower semi-modular, by \cite[Theorem 3]{gein}.
\end{proof}

\section{Modular Leibniz algebras}
A subalgebra $U$ of a Leibniz algebra $L$ is called {\em modular} in $L$ if the following two conditions hold:

\begin{align}
\langle U,V\rangle \cap W &= \langle V,U\cap W\rangle \hbox{ for all subalgebras } V,W\subseteq L \hbox{ with } V\subseteq W, \\ 
\langle U,V\rangle \cap W &= \langle V\cap W,U\rangle \hbox{ for all subalgebras } V,W\subseteq L \hbox{ with } U\subseteq W
\end{align}
In particular, quasi-ideals are modular. We call $L$ {\it modular} if every subalgebra of $L$ is modular in $L$. 

In the following result we establish when a cyclic Leibniz algebra is modular:

\begin{prop} A cyclic Leibniz algebra $L$ of dimension $n$ is modular if and only if it is one of the following two types:
\begin{itemize}
\item[(i)] nilpotent, so $L=\langle a\rangle$ where $[a^i,a]=a^{i+1}$ for $1\leq i\leq n-1$, and all other products are zero; or
\item[(ii)] solvable with $L=\langle a\rangle$ where $[a^i,a]=a^{i+1}$ for $1\leq i\leq n-1$, $[a^n,a]=a^n$, and all other products are zero.
\end{itemize}
\end{prop}
\begin{proof} Let $\{a,a^2,\ldots,a^n\}$ be a basis for $L$ with $[a^n,a]=\alpha_1 a+\cdots+\alpha_n a^n$. The Leibniz identity shows that $\alpha_1=0$. Let $T$ be the matrix for $R_a$ with respect to this basis, so that it is the companion matrix for $$p(x)=x^n-\alpha_nx^{n-1}-\cdots-\alpha_2x=p_1(x)^{n_1}\cdots p_s(x)^{n_s},$$ where the $p_j$ are the distinct irreducible factors of $p$ and $p_1(x)=x$. Let $L=W_1\dot{+}\cdots \dot{+}W_s$ be the associated primary decomposition of $L$ with respect to $R_a$. Then, as in \cite[Theorem 4.1]{batten}, we have
\begin{align}
[W_j,W_k] & =0 \hbox{ for } 2\leq j,k\leq s; \nonumber \\
[W_1,W_j]& =0 \hbox{ for } 2\leq j \nonumber; \\
[W_j,W_1]& \subseteq W_j \hbox{ for } 1\leq j\leq s. \nonumber
\end{align}
(Note that \cite{batten} concerns left Leibniz algebras, and there is a slight error in that paper in the first equation given here). Let $w \in W=W_2\oplus \cdots \oplus W_s$. Then $\langle W_1,w\rangle \cap W=\langle w,W_1 \cap W\rangle=\langle w\rangle$, and so $\langle w\rangle$ is an ideal of $L$. Thus $L$ has a basis $\{x_1, \ldots,x_{n_1},w_1,\ldots,w_k\}$ where $[x_i,a]=x_{i+1}$ for $1\leq i\leq n_1-1$, $[x_{n_1},a]=0$, $[w_i,a]=\lambda_i w_i$ for $1\leq i\leq k$. Then $$\lambda_iw_i+\lambda_jw_j=[w_i+w_j,a]=\lambda(w_i+w_j),$$ where $0\neq \lambda \in F$, since $\langle w_i+w_j\rangle$ is an ideal of $L$. Hence $\lambda_i=\lambda_j=\lambda$.
\par

Let $a= \sum_{i=1}^{n_1}\mu_ix_i+\sum_{i=1}^k\nu_iw_i$. Then
\begin{align}
a^2 & =\sum_{i=1}^{n_1-1}\mu_ix_{i+1}+\lambda\sum_{i=1}^k\nu_iw_i, \nonumber \\
 &{}\qquad \qquad \;\;\dots\nonumber \\
 a^{n_1} & = \mu_1x_{n_1}+\lambda^{n_1-1}\sum_{i=1}^k\nu_iw_i, \nonumber \\
a^{n_1+1} & = \lambda^{n_1}\sum_{i=1}^k\nu_iw_i \nonumber.
\end{align}
Since $L=\langle a\rangle$ we have $k=0$ or $1$. Replacing $a$ by $(1/\lambda) a$ we have the multiplication given in (i) or (ii).
\par

In case (i) all of the subalgebras are inside $\Leib(L)$, as in Lemma \ref{l:nilp}, and so it is easy to check that they are modular. Then suppose that case (ii) holds. Let $U$ be a subalgebra which is not in $\Leib(L)$, and let $u=a+\sum_{i=2}^n\lambda_ia^i\in U$. Then
\begin{align}
u^2 & = a^2+\sum_{i=2}^{n-1}\lambda_ia^{i+1}+\lambda_na^n, \nonumber \\
&{}\qquad \qquad \;\;\dots\nonumber \\
u^{n-1} & = a^{n-1}+(\sum_{i=2}^n\lambda_i)a^n, \nonumber \\
u^n & = (1+\sum_{i=2}^n\lambda_i)a^n \nonumber. 
\end{align}
Hence $u^n-u^{n-1}=a^n-a^{n-1}\in U$. But then $$u^{n-1}-u^{n-2}=a^{n-1}-a^{n-2}+\lambda_2(a^n-a^{n-1})\in U,$$ so $a^{n-1}-a^{n-2}\in U$. Similarly we have that $a^i-a^{i-1}\in U$ for $2\leq i\leq n$. Thus $U$ has codimension 1 in $L$ and so is a quasi-ideal and hence modular. All other subalgebras of $L$ are inside $\Leib(L)$ and modularity is straightforward to check.
\end{proof}

The next result shows that the extraspecial Leibniz algebras described in Proposition \ref{p:esp} are indeed modular.

\begin{prop}\label{p:esp2} An extraspecial Leibniz algebra $L$ is modular if and only if $J=\{x\in L \mid x^2=0\}$ is an abelian ideal of $L$. 
\end{prop}
\begin{proof} If $L$ is modular then it is upper semi-modular, so the necessity follows from Proposition \ref{p:esp}.
\par

Conversely, suppose that $L$ is extraspecial and $J=\{x\in L \mid x^2=0\}$ is an abelian ideal of $L$. It follows from the fact that modular nilpotent Lie algebras are abelian that every subalgebra of $L$ is either inside $J$ or is an ideal of $L$. The modular identities are then straightforward to check.
\end{proof}

In \cite{qi} Towers defined a subspace $U$ of a Leibniz algebra $L$ to be a {\em quasi-ideal} if $[U,V]+[V,U]\subseteq U+V$ for every subspace $V$ of $L$. Similarly we shall define a subalgebra $U$ of $L$ to be a {\em weak quasi-ideal} if $[U,V]+[V,U]\subseteq U+V$ for every subalgebra $V$ of $L$. Then we have the following result.

\begin{prop}\label{p:qi} Let $L$ be a Leibniz algebra over an algebraically closed field. The following conditions are equivalent:
\begin{itemize}
\item[(i)] $L$ is modular; 
\item[(ii)] every subalgebra of $L$ is a weak quasi-ideal of $L$; and
\item[(iii)] $[x,y]\in \langle x\rangle  +\langle y\rangle  $ for all $x,y\in L$.
\end{itemize}
\end{prop}
\begin{proof} (i) $\Rightarrow$ (ii): Let $U,V$ be subalgebras of $L$. Since $L$ is modular, the intervals $[U:\langle U,V\rangle  ]$ and $[U\cap V:V]$ are isomorphic as lattices. As $L$ is supersolvable (by Corollary \ref{usmac}) this implies that $\dim \langle U,V\rangle  - \dim U=\dim V - \dim U\cap V$, whence
\[ \dim \langle U,V\rangle  =\dim U + \dim V - \dim U\cap V = \dim (U+V).
\] It follows that $\langle U,V\rangle  =U+V$ and $U$ is a weak quasi-ideal of $L$.
\par

(ii) $\Rightarrow$ (i): If we assume (ii) then the two modular identities (1) and (2) are easily checked.
\par

(ii) $\Rightarrow$ (iii): If (ii) holds, then $[x,y]\in \langle x,y\rangle  =\langle \langle x\rangle  ,\langle y\rangle  \rangle  \subseteq \langle x\rangle  +\langle y\rangle  $.
\par

(iii) $\Rightarrow$ (ii): Suppose that (iii) holds and let $U,V$ be subalgebras of $L$. Then $[u,v], [v,u]\in U+V$ for all $u\in U, v\in V$, whence $U$ is a weak quasi-ideal of $L$.
\end{proof}

\begin{remark}\label{rem} \emph{In the proof of Proposition \ref{p:qi}, the assumption that the ground field is algebraically closed is only used in the implication ``(i) $\Rightarrow$ (ii)". Therefore, the remaining implications remain valid over arbitrary fields.} 
\end{remark}

\begin{coro}\label{extras} Let $L$ be a Leibniz algebra over a field $F$. Suppose that $L=E\oplus Z$, where $Z$ is a central ideal of $L$ and  $E$ is an extraspecial Leibniz algebra such that $J=\{x\in L \mid x^2=0\}$ is an abelian ideal of $L$. Then $L$ is modular.
\end{coro}
\begin{proof} Note that every subalgebra of $L$ is either inside $J+Z$ or is an ideal of $L$. As a consequence, if $U$ and $V$ are subalgebras of $L$, then one has that $[U,V]+[V,U]\subseteq U+K$, and the conclusion follows from Proposition \ref{p:qi} and Remark \ref{rem}. 
\end{proof}

Notice that the algebras described in Corollary \ref{extras} include those in which every subalgebra is an ideal, as described in \cite{kss}; they are the ones for which $J=\Leib(L)$.
\par

For our next result we shall need the following from \cite{batten}, which we include for the reader's convenience.

\begin{theor}\label{t:batten} (\cite[Theorem 2.5]{batten}) Let $L$ be a four-dimensional nonsplit non-Lie nilpotent Leibniz algebra with $\dim(L^2)=2=\dim(\Leib(L))$ and $\dim(L^3)=0$. Then, L is isomorphic to a Leibniz algebra spanned by $x_1,x_2,x_3,x_4$ with the nonzero products given by the following:
\begin{itemize}
\item $A_{14}: [x_1,x_1] =x_3,[x_1,x_2] =x_4$;
\item $A_{15}: [x_1,x_1] =x_3,[x_2,x_1] =x_4$;
\item $A_{16}: [x_1,x_2] =x_4,[x_2,x_1] =x_3,[x_2,x_2] = −x_3$;
\item $A_{17}: [x_1,x_1] =x_3,[x_1,x_2] =x_4,[x_2,x_1] =\alpha x_4, \alpha \in \C\setminus\{−1, 0\}$; 
\item $A_{18}: [x_1,x_1] =x_3,[x_2,x_1] =x_4,[x_1,x_2] =\alpha x_3,[x_2,x_2] = −x_4,\alpha \in \C\setminus \{−1\}$; or
\item $A_{19}: [x_1,x_1] =x_3,[x_1,x_2] =x_3,[x_2,x_1] =x_3+x_4,[x_2,x_2] =x_4$.
\end{itemize}
\end{theor}

In \cite{batten} the authors are assuming that the algebras are defined over the complex field.  However, all the basis changes made in the proof of this theorem are valid over any algebraically closed field of characteristic different from $2$; the only place where $2$ is a problem is in the choice of $s$ in the penultimate line. By {\em nonsplit} they mean that $L$ is not a direct sum of two non-zero ideals.

\begin{prop}\label{p:class2} Let $L$ be a Leibniz algebra over an algebraically closed field of characteristic different from $2$ and suppose that $L^3=0$. Then $L$ is modular if and only if it is of the form given in Corollary \ref{extras}.
\end{prop}
\begin{proof} Suppose first that $L$ is modular. We use induction on the (minimal) number of generators of $L$ as a Leibniz algebra. If $L$ is cyclic, then it is at most two-dimensional and so is clearly of the claimed form. Hence, suppose that $L$ is generated by two elements, $x,y$, say. If $L$ has dimension $3$ or less, then it is clearly of the claimed form; if not, it must have dimension $4$.
\par

A four-dimensional nilpotent modular Leibniz algebra must have a basis $x, x^2, y, y^2$ with $[x,y] = \alpha_1 x^2 + \alpha_2 y^2$ and $[y,x] = \beta_1 x^2 + \beta_2 y^2$, by Proposition \ref{p:qi}. Then $L$ must be one of the algebras given in Theorem \ref{t:batten}. None of these are modular, as is shown below.
\begin{align*}
A_{14}:\hspace{.5cm} & [x_1,x_2]=x_4\notin \langle x_1\rangle  +\langle x_2\rangle  =Fx_1+Fx_2+Fx_3; \\
A_{15}:\hspace{.5cm} & [x_2,x_1]=x_4\notin \langle x_1\rangle  +\langle x_2\rangle  =Fx_1+Fx_2Fx_3; \\
A_{16}:\hspace{.5cm} & [x_1,x_2]=x_4\notin \langle x_1\rangle  +\langle x_2\rangle  =Fx_1+Fx_2+Fx_3; \\
A_{17}:\hspace{.5cm} & [x_1,x_2]=x_4\notin \langle x_1\rangle  +\langle x_2\rangle  =Fx_1+Fx_2+Fx_3; \\
A_{18}:\hspace{.5cm} & [x_1-x_2,x_1]=x_3-x_4\notin \langle x_1-x_2\rangle  +\langle x_1\rangle  =Fx_1+Fx_2+Fx_3;  \\
A_{19}:\hspace{.5cm} & [x_1-x_2,x_1]=-x_4\notin \langle x_1-x_2>+\langle x_1\rangle =Fx_1+Fx_2+Fx_3.  
\end{align*}
That just leaves the split case. If $L$ is a direct sum of a three-dimensional ideal and a one-dimensional ideal, then the latter is in the centre and so this is of the claimed form. If it is a direct sum of two two-dimensional ideals, then it is of the form $(Fx+Fx^2) \oplus (Fy+Fy^2)$ and $[x+y,x-y] = x^2 - y^2$, which cannot be written as a direct sum of $(x+y)^2 = x^2 + y^2$ and $(x-y)^2 = x^2 + y^2$.
\par

Assume now that the result is true for $n \geq 2$, and suppose that $L$ is generated by $n+1$ elements. Then $L = L_1 + Fx$ where $L_1$ is of the form $E \oplus C$ with $E$ being an extraspecial Leibniz algebra and $C$ being a central ideal of $L$. By considering pairs of generators and using the above we see that $L^2 = (L_1)^2 = Fz$, say. Suppose there exists $y \in L_1 \setminus Z(L_1)$ such that $x+y \in Z(L)$. Then $L = L_1 \oplus F(x+y)$ is of the required form. Similarly, $L = L_1 \oplus Fx$ if $x \in Z(L)$. If neither, then $L$ is of the form $\tilde{E} \oplus \tilde{C}$ where, $\tilde{E}$ is an extraspecial Leibniz algebra, $\tilde{C}$ is a central ideal of $L$, and $Z(\tilde{E}) = Fz = \tilde{E}^2$, which completes the proof.
\end{proof}

Extending the above result to the case where $L^n=0$ and $n>3$ is far from straightforward. It is easy to see that any Leibniz algebra of the form $E\oplus C$, where $E$ is extraspecial in which $J$ is an abelian ideal, or a nilpotent cyclic algebra and $C$ is a central ideal, is modular. However, not every nilpotent modular Leibniz algebra $L$ has this form, even if $L$ is four-dimensional, as the following example shows:

\begin{ex} \emph{Let $L$ have basis $x_1,x_2,x_3,x_4$ and nonzero products $[x_1,x_1] =x_3,[x_2,x_2] =x_4,[x_1,x_3] =x_4$ (This is $A_{25}$ in \cite[Theorem 2.6]{batten}). Then it is straightforward to check that $L$ is modular but is not of the form given in the previous paragraph.}              
\end{ex}

It would also be good to remove the requirement of an algebraically closed field in Proposition \ref{p:class2}. However, this is not straightforward either, as modularity is not preserved by extending the base field, as the following example shows.

\begin{ex} \emph{Let $L$ be the extraspecial Leibniz algebra $L=Fx+Fy+Fz$ with $[x,y]=z, [y,z]=-z,  x^2=y^2=z$  and $z^2=0$, the other products being zero. By Proposition \ref{p:qi} and Remark \ref{rem}, it is easy to see that $L$ is modular over the real number field and non-modular over the complex number field. (For the latter conclusion, note that the elements $x+iy$ and $x-iy$ are in $J$ but their sum is not.) }
\end{ex}

\end{document}